\documentclass[12pt]{amsart}
\newcounter{defcounter}
\usepackage{amsmath,amssymb}
  \usepackage[all]{xy}
  \usepackage{pdfsync}
  \usepackage[mathscr]{euscript}

\theoremstyle{plain}
\newtheorem{theorem}{Theorem}

\newtheorem{proposition}[theorem]{Proposition}

\newtheorem{lemma}[theorem]{Lemma}

\newtheorem{proposition.definition}[theorem]{Proposition/Definition}

\newtheorem{theoremalpha}{Theorem}

\theoremstyle{definition}

\newtheorem{remark}[theorem]{Remark}
\newtheorem{example}[theorem]{Example}

\newcommand{\lra}{\longrightarrow}

\newcommand{\noi}{\noindent}
\newcommand{\PP}{\mathbf{P}}

\newcommand{\CC}{\mathbf{C}}

\newcommand{\OO}{\mathcal{O}}

\newcommand{\JJ}{\mathcal{J}}

\newcommand{\II}{\mathcal{I}}

\newcommand{\frakm}{\mathfrak{m}}

\newcommand{\eps}{\varepsilon}

\newcommand{\image}{\textnormal{im}}

\newcommand{\HH}[3]{H^{{#1}} \big( {#2} , {#3}
\big) }
\newcommand{\HHstar}[3]{H_{*}^{{#1}} \big( {#2} , {#3}
\big) }

\newcommand{\hh}[3]{h^{{#1}} \big( {#2} , {#3}
\big) }

\newcommand{\coker}{\textnormal{coker}}

\newcommand{\Div}{\text{Div}}

\newcommand{\lin}{\equiv_{\text{lin}}}

\newcommand{\Linser}[1]{| \mspace{1 mu} {#1}
\mspace{1 mu} |}
\newcommand{\linser}[1]{\Linser{  {#1}  }}

\newcommand{\pro}{\textnormal{pr}}

\newcommand{\Sym}{\textnormal{Sym}}

\newcommand{\Hom}{\textnormal{Hom}}

\numberwithin{theorem}{section}

\begin{document}

\title{Torelli Theorems for Some Steiner Bundles}

 \author{Robert Lazarsfeld}
  \address{Department of Mathematics, Stony Brook University, Stony Brook, New York 11794}
 \email{{\tt robert.lazarsfeld@stonybrook.edu}}
 \thanks{Research  of the first author partially supported by NSF grant DMS-1739285.}

 \author{John Sheridan}
  \address{Department of Mathematics, Princeton University, Princeton, New Jersey  08544}
 \email{{\tt john.sheridan@princeton.edu}}

\maketitle

\newcommand{\Val}{\textnormal{Vall\`es}}
\newcommand{\Steiner}{\textnormal{Steiner}}
\newcommand{\Spec}{\textnormal{Spec}}

\section*{Introduction}

A Steiner bundle is a vector bundle $E$ on projective space $\PP^r$ that sits in a short exact sequence
\begin{equation} \label{Steiner.SES.Intro} 0 \lra U_1 \otimes \OO_{\PP^r}(-1) \overset{\phi}{\lra} U_0 \otimes \OO_{\PP^r}\lra E \lra 0,\end{equation}
where $U_1$ and $U_0$ are finite-dimensional vector spaces. These bundles arise in several geometric settings, and by now they are  the focus of a  substantial literature (c.f.\! \cite{AnconaOttaviani, Arrondo, ArrondoMarchesi, Dolgachev.Kapranov, Ein.NB, Ellia.ea, Valles1, Valles2, Huizenga}).  When $\PP^r = \linser{V}$ is the projective space of one-dimensional subspaces of a vector space $V$, $\phi$ is given by a linear map
\begin{equation}  \label{Steiner.Mult.Intro.Eqn}
\mu : U_1 \otimes V \lra U_0 \end{equation}
having the property that $\mu(u_1\otimes v) \ne 0$ for all non-zero   $u_1 \in U_1, v \in V$.  For example, suppose that $V \subseteq \HH{0}{X}{A}$ is a very ample linear series on a smooth complex projective variety $X$ of dimension $n$. If $B$ is a sufficiently positive line bundle on $X$, then the Steiner  bundle corresponding to multiplication
\begin{equation} \label{Mult.Map.LBs.Intro}
\mu : \HH{0}{X}{B \otimes A^*} \otimes V \lra \HH{0}{X}{B} 
\end{equation}
is the tautological vector bundle $E_{|V|, B}$ on $\linser{V}$ whose fibre at $[s] \in \linser{V}$ is the vector space  $\HH{0}{X}{B \otimes \OO_{\Div(s)}}$. For  $X = \PP^1$ and $V = \HH{0}{\PP^1}{\OO_{\PP^1}(r)}$, these are known as Schwarzenberger bundles. The bundles $E_{|V|,B}$ were considered by Arrondo  in \cite{Arrondo} from a somewhat different perspective.

In their influential paper \cite{Dolgachev.Kapranov}, Dolgachev and Kapranov consider the  bundle $E =  \Omega^1_{\PP^r}  ( \log \Sigma H_i  )$ of logarithmic forms with poles along  a normal-crossing hyperplane arrangement  on $\PP^r$. 
They show that $E$ is a Steiner bundle, and they establish moreover that one can recover the arrangement from $E$ provided that the planes do not osculate a rational normal curve. This is the prototype of a Torelli-type statement, asserting that a Steiner bundle determines the geometric data used to construct it. Other results along these lines appear in the papers \cite{Arrondo,    Angelini,  BallicoHuh, Dolgachev, Faenzi.ea, Ueda, Valles1, Valles2}.

The purpose of this note is to point out that similar Torelli theorems  hold for the tautological bundles $E_{|V|, B}$ once  $B$ is positive enough. To begin  with, we prove 
\begin{theoremalpha} \label{Intro.Thm.A}
Let $X$ be a smooth projective variety, and let  $V \subseteq \HH{0}{X}{A}$ be a very ample linear series. Fix a line bundle $B$ on $X$ that satisfies
\begin{equation} \label{Van.Hypothesis.Theorem.A} \HH{i}{X}{B \otimes A^{\otimes -(i+1)}} \ = \ 0 \ \ \text{ for } \ \ i > 0. \notag \end{equation}
Then one can recover the embedding $X \subseteq \PP(V)$, together with the line bundle $B$,  from the Steiner bundle  associated to the multiplication mapping \eqref{Mult.Map.LBs.Intro}.
\end{theoremalpha}
\noi In the event that $V$ is basepoint-free but possibly not very ample, one can recover from $E$ the image $\phi_{|V|}(X) \subseteq \PP(V)$.  The Theorem gives a partial answer to Question 0.2 of Arrondo's paper \cite{Arrondo}. 

Observe that if $ 
B= \OO_X(K_X + (n+2)A )$, then the hypothesis of the Theorem is satisfied automatically  thanks to Kodaira vanishing. It is natural to ask what happens for slightly less positive $B$, for instance $B = \OO_X(K_X + (n+1)A)$. The example where $X  = X_d \subseteq \PP^{n+1}$ is a smooth hypersurface of degree $d$, $A = \OO_X(1)$ and $V = \HH{0}{X}{A}$ is instructive here.  In this case $B = \OO_X(d-1) $, and therefore  the map $\mu$ in  \eqref{Mult.Map.LBs.Intro} is identified with 
 multiplication \[
\HH{0}{\PP^{n+1}}{\OO_{\PP^{n+1}} (d-2)} \otimes \HH{0}{\PP^{n+1}}{\OO_{\PP^{n+1}}(1)}\lra \HH{0}{\PP^{n+1}}{\OO_{\PP^{n+1}}(d-1)}.
\]
of all homogeneous polynomials of the indicated degrees. In particular, $\mu$ -- and hence also $E_{|V|, B}$ -- doesn't see $X$. Amusingly, it turns out that this is the only situation in which Torelli fails  for the tautological bundles associated to  $\OO_X(K_X + (n+1)A)$.

Specifically, we use considerations of Koszul cohomology and Green's $K_{p,1}$ Theorem from \cite{Koszul1} to prove:
\begin{theoremalpha} \label{Intro.Thm.B}
Let $V = \HH{0}{X}{A}$ for a very ample divisor $A$ on $X$, and assume that $A$ does not embed $X$ as a hypersurface in $\PP^{n+1}$. $($In particular, we suppose that $(X, A) \ne (\PP^n, \OO_{\PP^n}(1))$.$)$
\begin{itemize}
\item[$(i).$] If $B = \OO_X(K_X + (n+1)A)$, then $E_{|V|, B}$ determines $X$. \vskip 8pt
\item[$(ii).$] Assume that $\deg_A(X) \ge \dim \HH{0}{X}{A} + 2 -n,$ and that $\HH{1}{X}{\OO_X}= 0$  when $n \ge 2$. Then the same conclusion holds for \[  B = \OO_X(K_X + nA)\] except when $\phi_{|A|}(X) \subseteq \PP(V) $ lies on an $(n+1)$-fold of minimal degree.\end{itemize}
\end{theoremalpha}
\noi 
(In the exceptional case of (ii), the bundles $E_{|V|, B}$ only depend on the   scroll containing $X$: see Example \ref{Divisors.in.Scrolls.Exampe}.)
Finally, we show   that   these ideas  lead to a new proof of the theorem of Dolgachev and Kapranov.

Our arguments revolve around a strategy pioneered by Vall\`es \cite{Valles1,  Valles2}. His   idea is to study hyperplanes $\linser{W} \subseteq \linser{V}$ for which the restriction $E\, |_{|W|}$ has a trivial quotient, with the aim of recovering $X \subseteq \PP(V)$ as the locus of all such. In \S1, we show that Theorem \ref{Intro.Thm.A}  follows easily   from considerations of Castelnuovo--Mumford regularity. For Theorem \ref{Intro.Thm.B}, in \S 2, we use duality to relate the existence of unstable hyperplanes to the non-vanishing of certain Koszul cohomology groups, where Green's results apply. In the Appendix, we indicate how to recover the result of  Dolgachev--Kapranov.

We work throughout over the complex numbers, although this hypothesis isn't needed for Theorem \ref{Intro.Thm.A}. Given a vector space $V$, $\PP(V)$ denotes the projective space of one-dimensional quotients of $V$, while $\linser{V} = \PP(V^*)$ is used for the space of one-dimensional subspaces.   Somewhat sloppily, we take the liberty of freely confounding divisors and line bundles. We are grateful to Igor Dolgachev for valuable comments.

\numberwithin{equation}{section}

\section{Proof of Theorem \ref{Intro.Thm.A}}

We start by fleshing out the construction of Steiner bundles indicated in the Introduction. Fix a linear map
\begin{equation} \label{Mult.Section.1}  \mu : U_1 \otimes V \lra U_0, \end{equation}
where $U_1, U_0$ and $V$ are finite-dimensional complex vector spaces with $\dim V = r+1$. Composing with the canonical inclusion
$ \OO_{|V|}(-1)  \subseteq   V \otimes_{\CC} \OO_{|V|}$ 
of vector bundles on the projective space $\linser{V}$, $\mu$ gives rise to a morphism  
 \[ \phi:  U_1 \otimes \OO_{|V|}(-1)  \lra  U_0 \otimes \OO_{|V|} \]  of locally free sheaves. Assume now that $\mu(u_1 \otimes v) \ne 0$ for all non-zero vectors $u_1 \in U_1$ and $v \in V$. Then $\phi$ is injective of constant rank, and therefore $E  =  \coker \, \phi$ is a vector bundle that sits in an exact sequence
 \begin{equation} \label{Steiner.Section.1}
 0 \lra U_1 \otimes \OO_{|V|}(-1) \overset{\phi}  \lra  U_0 \otimes \OO_{|V|} \lra E \lra 0. 
 \end{equation}
 We will sometimes write $E = \Steiner(\mu)$ when we wish to emphasize the role of $\mu$. It is elementary that conversely every Steiner bundle $E$ arises in this fashion.

The main example  for our purposes are   tautological bundles associated to a linear system of divisors. Let $X$ be a smooth projective variety of dimension $n$, and let $A$ be a very ample (or at least ample and  basepoint-free) line bundle on $X$. Fix a very ample (or basepoint-free) subspace $V \subseteq \HH{0}{X}{A}$ of dimension $r+1$, and denote by
\[
\mathcal{D} \, \subseteq  \, X \times \linser{V} 
\]
the universal divisor, consisting of pairs $(x, [s])$ such that $s(x) = 0$. It is realized as the zero-locus of a section of $A \boxtimes \OO_{|V|}(1)$. Now consider a line bundle $B$ on $X$ which is sufficiently positive so that $\HH{1}{X}{B \otimes A^*} = 0$, and set 
\[  E_{|V|,B} \ = \ \pro_{2,*} \big ( \pro_{1}^* B \otimes \OO_{\mathcal{D}} \big). \]
This is a vector bundle on $\linser{V}$ whose fibre at $[s]$ is identified with the space  $\HH{0}{X}{B \otimes \OO_{\Div(s)}}$ of sections of the restriction of $B$  to the divisor $\{ s = 0 \}$. 
Starting with the exact sequence
\[
0 \lra \pro_1^* (B\otimes A^*) \otimes \pro_2^* \OO_{|V|}(-1)\overset{\cdot \mathcal{D}} \lra \pro_1^*B \lra \pro_1^* B \otimes \OO_{\mathcal D} \lra 0 
\]
on $X \times \linser{V}$ and pushing forward to $\linser{V}$, one finds:
\begin{lemma}
Assuming always that $\HH{1}{X}{B \otimes A^*} = 0$, $E_{|V|,B}$ is the Steiner bundle on $\linser{V}$ determined by the natural multiplication map
\[
\HH{0}{X}{B \otimes A^*} \otimes V \lra \HH{0}{X}{B}. \ \ \qed
\]
\end{lemma}
\noi  We remark that these statements remain true without change if $B$ is a vector bundle of higher rank on $X$.

Returning to the general setting of \eqref{Mult.Section.1}, fix a codimension one subspace $W \subseteq V$  defining a hyperplane $\linser{W} \subseteq \linser{V}$. One says that $\linser{W}$ is an \textit{unstable plane} for the Steiner bundle $E$   if the restriction of $E$ to $\linser{W}$ has a trivial quotient, i.e.\! if
\[
\HH{0}{\linser{W}} { E^*  |_{|W|}} \ \ne \ 0. 
\]
Hyperplanes in $\linser{V}$ correspond to  points in $\PP(V)$, and we define the \textit{Vall\`es locus} of $E$ to be the algebraic subset $\Val(E) \subseteq \PP(V)$ parameterizing all unstable planes.

The following remark is elementary but crucial:
\begin{lemma} \label{Key.Lemma.Section.1}
A hyperplane $W \subseteq V$ corresponds to a point  in the Vall\`es locus of a Steiner bundle $E = \Steiner(\mu)$ if and only if 
the restriction
\[ \mu_{| U_1 \otimes W}:  U_1 \otimes W\lra  U_0 \]
of $\mu$  to $U_1 \otimes W \subseteq U_1 \otimes V$ fails to be surjective. 
\end{lemma}
\begin{proof}
In fact, one sees using  the exact sequence
\[
 0 \lra U_1 \otimes \OO_{|W|}(-1)   \lra  U_0 \otimes \OO_{|W|} \lra E|_{|W|}\lra 0 \]
that $\Hom ( E_{|W|}, \OO_{|W|}  ) \cong \coker  (\mu_{| U_1 \otimes W})^*$.
\end{proof}

We now move towards the proof of Theorem \ref{Intro.Thm.A}. Let $X$ be a smooth projective variety of dimension $n$, and $A$ a very ample (or ample and globally generated) line bundle on $X$. Fix a line bundle $B$ on $X$ that satisfies the vanishings
\begin{equation} \label{Van.Hypothesis.Section.1}
\HH{i}{X}{B \otimes A^{\otimes{-(i+1)}}} \ = \ 0 \ \ \text{ for } i > 0. 
\end{equation}
This implies that $B$ is very ample (e.g. by \cite[Example 1.8.22]{PAG} or via the arguments below). 

The Theorem is essentially a consequence of the following
\begin{proposition} \label{BPF.Proposition.Section.1}
Keeping the hypotheses just stated, let $U \subseteq \HH{0}{X}{A}$ be a subspace.
\begin{itemize}
\item[$(i)$.] If $U$ is basepoint-free, then the multiplication mapping
\[\mu_U :  \HH{0}{X}{B \otimes A^*} \otimes U \lra \HH{0}{X}{B} \]
is surjective.
\vskip 5pt
\item[(ii).] Suppose that $U$ generates the sheaf $A \otimes \frakm_x$ of sections of $A$ vanishing at some point $x \in X$ $($so that in particular every section in $U$ vanishes at $x)$.  Then 
\[  \image(\mu_U) \, = \, \HH{0}{X}{B \otimes \frakm_x}. \]
\end{itemize}
\end{proposition}

\begin{proof}
This follows from the theory of Castelnuovo--Mumford regularity (cf \cite[Section 1.8]{PAG}), but for the convenience of the reader we sketch briefly the argument. Assuming $U$ is basepoint-free, it determines a surjective mapping $U \otimes A^* \lra \OO_X $ of bundles on $X$. The resulting Koszul complex yields  a long exact sequence
\[  \ldots   \lra \Lambda^3 U \otimes A^{\otimes -3}\lra \Lambda^2 U \otimes A^{\otimes -2} \lra U \otimes A^* \lra \OO_X \lra 0  \]
of vector bundles on $X$. Tensoring through by $B$ and taking cohomology, the hypothesis \eqref{Van.Hypothesis.Section.1} implies with a diagram chase that the map
\[ \HH{0}{X}{U \otimes B \otimes A^*} \lra \HH{0}{X}{B} \]
is surjective. This proves (i). In the setting of (ii), $U \otimes A^*$ maps onto $\frakm_x$, and now one arrives at a complex having the shape:\[
\ldots   \lra \Lambda^3 U   \otimes A^{\otimes -3}\lra \Lambda^2 U  \otimes A^{\otimes -2} \lra U \otimes   A^* \overset{\eps} \lra  \frakm_x \lra 0,\]
with $\eps$ surjective. This complex is not exact, but its homology sheaves are supported at the point $x$, and another diagram chase shows that this suffices to conclude the surjectivity of \[ \HH{0}{X}{U \otimes B \otimes A^*} \lra \HH{0}{X}{B \otimes \frakm_x}.\] We refer to \cite[Example B.1.3]{PAG} for more details.
\end{proof}

\begin{proof} [Proof of Theorem \ref{Intro.Thm.A}]  Assume that $V \subseteq \HH{0}{X}{A}$ is very ample, and that \eqref{Van.Hypothesis.Section.1} holds. We assert to begin with that
\[   \Val(E_{|V|, B})\ = \ \phi_{|V|}(X)  \ \subseteq \ \PP(V). \tag{*} \]
In fact, by definition points in the image of $\phi_{|V|}$ correspond to hyperplanes $W \subset V$ consisting of sections that vanish at a fixed point $x\in X$. On the other hand,  the preceeding Proposition shows that  these are precisely the hyperplanes $W$ for which
\[  \HH{0}{X}{B \otimes A^*} \otimes W \lra \HH{0}{X}{B} \]
fails to be surjective. So (*) follows from Lemma \ref{Key.Lemma.Section.1}.  It remains to show that one can recover the line bundle $B$ from $E_{|V|,B}$. Suppose then that $W \subseteq V$ is a hyperplane corresponding to a point $\phi_{|V|}(x) \in \PP(V)$.  Proposition \ref{BPF.Proposition.Section.1} (ii) and (the proof of) Lemma  \ref{Key.Lemma.Section.1} imply 
 that $E|_{|W|}$ has a unique trivial quotient. Via the isomorphism
\[  \HH{0}{X}{B} \, = \,   \HH{0}{\linser{W}}{E|_{|W|} }\]
this determines a one-dimensional   quotient of $H^0(B)$.  One verifies that this is $\phi_{|B|}(x) \in \PP H^0(B)$, and therefore one can reconstruct from $E$ the embedding defined by $B$, as claimed. 
\end{proof}

\begin{remark}
We leave it to the reader to check that the equality 
\[  \Val(E_{|V|,B}) = \phi_{|V|}(X) \subseteq \PP(V)\]
 still holds  assuming only that  $V \subseteq \HH{0}{X}{A}$ is  basepoint-free. We note also that everything we have said goes through with only  evident minor changes if $B$ is a vector bundle of higher rank, provided of course that \eqref{Van.Hypothesis.Section.1}
still holds. \qed
\end{remark}

\section{Proof of Theorem \ref{Intro.Thm.B}}

Theorem \ref{Intro.Thm.B} draws on some ideas and results concerning Koszul cohomology. We start by recalling  the requisite definitions and facts.
 
Let $X$ be a smooth complex projective variety of dimension $n$, $A$ a very ample line bundle on $X$, and $N$ an arbitrary line bundle. Fix a basepoint-free subspace $U \subseteq \HH{0}{X}{A}$. For every $p, q  \ge 0$ one can form the Koszul-type   complex

\vskip -15pt
\footnotesize
\[
\ldots\lra  \Lambda^{p+1}U \otimes H^0(N \otimes A^{\otimes (q-1)}) \lra \Lambda^{p}U \otimes H^0(N \otimes A^{\otimes q}) \lra \Lambda^{p-1}U \otimes H^0(N \otimes A^{\otimes (q+1)}) \lra \ldots .
\]
\normalsize

\vskip -5pt
\noi  
The cohomology of this complex is denoted by $$K_{p,q}(X, N; U),$$   the line bundle $A$ being understood; when $N = \OO_X$ we write simply $K_{p,q}(X; U)$.  These vector spaces control the  minimal free resolution of $\oplus_k\,  \HH{0}{X}{N \otimes A^{\otimes k}}$ viewed as a graded module over the symmetric algebra $\Sym(U)$, but  we do not draw on this interpretation. However  this is a very rich and well-developed story, and we refer for instance to \cite{Koszul1, AproduNagel, LSGAV} for introductions.

Theorem \ref{Intro.Thm.B} is a simple consequence of several observations and results about these groups. The first follows immediately from the definitions: 
\begin{lemma} \label{Elem.Koszul.Lemma|} Keep notation as above.
\begin{itemize}
\item[$(i)$.] If $p > 0$, then $K_{p,0}(X; U) = 0$.
\vskip 7pt
\item[$(ii)$.] Suppose that $W \subseteq U \subseteq \HH{0}{X}{A}$ is a codimension one basepoint-free subspace of $U$. Then there is a long exact sequence
\vskip -5pt
\footnotesize
\[  
\ldots \lra K_{p,q-1}(X,N;W) \lra K_{p,q}(X, N; W) \lra K_{p,q}(X,N; U) \lra K_{p-1,q}(X,N;W) \lra \ldots  \]
\normalsize
\end{itemize}
\end{lemma}
\begin{proof}
 The map $\Lambda^p U \lra \Lambda^{p-1}{U} \otimes H^0(A)$   is injective provided that $p > 0$, and by definition  $K_{p,0}(X; U)$ is its kernel. For (ii), the   exact sequence of vector spaces
\[
0 \lra \Lambda^p W \lra \Lambda^p U  \lra \Lambda^{p-1} W \lra 0
\]
determines a short exact sequence of the complexes computing Koszul cohomology. This gives rise to the stated long exact sequence  of cohomology groups. 
\end{proof}

The next point is Serre--Grothendieck duality for Koszul groups:
\begin{proposition} \label{Duality.for.Koszul.Groups}
Set $s = \dim\linser{U}$, and assume that \begin{equation} \label{Van.for.Duality.Eqn}
\HH{i}{X}{N \otimes A^{\otimes{(q-i)} }} \ = \ \HH{i}{X}{N \otimes A^{\otimes{(q-i-1)} }} 
\ = \ 0 \ \ \text{for \ } \ 0 < i < n.
\end{equation}
Then there are isomorphisms
\[
K_{p,q}\big(X,N; U\big) \ \cong \ K_{s-n-p, n+1-q}\big( X, \omega_X\otimes N^* ; U\big)^*,\]
where $\omega_X = \OO_X(K_X)$ is the canonical bundle of $X$. 
\end{proposition}
\noi This is established in  \cite[Theorem 2.c.6]{Koszul1} or \cite[Remark 2.26]{AproduNagel}.

The most interesting input to Theorem \ref{Intro.Thm.B} is: 
\begin{theorem}[Green's $K_{p,1}$ Theorem]  \label{Green.Kp1.Thm}
Set $V = \HH{0}{X}{A}$,  put $r = \dim \linser{V}$, and assume that $\deg_A(X)  \ge r - n + 3$. Then
\[ K_{r-n-1,1}(X; V) \ \ne \ 0 \]
if and only if $\phi_{|A|}(X) \subseteq \PP^r$ lies on an $(n+1)$-fold of minimal degree. 
\end{theorem}
\noi This is   \cite[Theorem 3.c.1]{Koszul1}; see also \cite[Chapter 3]{AproduNagel}  for a somewhat different approach.

With these preliminaries out of the way we return to the setting of Theorem \ref{Intro.Thm.B}. As before $X$ is a smooth complex projective variety of dimension $n$, and $A$ is a very ample line bundle on $X$. We will work with the full space of sections $V = \HH{0}{X}{A}$, and following traditional notation  we  write $\linser{A}$  for the corresponding complete linear series, i.e.\! $\linser{A} = \linser{ \HH{0}{X}{A}}$. We put $r = \dim \linser{A} = \hh{0}{X}{A}-1$.

Our goal  is to study the Torelli problem for the tautological Steiner bundles $E_{|A|,B}$   when 
\begin{equation} \label{Possibilities.for.B}
 B  \, = \, K_X + (n+1)A \ \ \text{or } \ \ B \, = \, K_X + nA. \end{equation}
The first point to observe is that these divisors are sufficiently positive to guarantee that every point of $X \subseteq \PP^r$ gives rise to an unstable plane.   \begin{lemma} Assume that $(X, A) \ne \big(\PP^n, \OO_{\PP^n}(1)\big)$. Then the divisor $K_X + (n+1)A$ is very ample and $K_X + nA$ is basepoint-free. Consequently, with $B$ as in \eqref{Possibilities.for.B} one has
\[
 \phi_{|A|}(X) \ \subseteq \ \Val( E_{|A|,B} ) \ \subseteq \ \PP \HH{0}{X}{A}. 
 \]
 \end{lemma}
 \begin{proof} The first assertion  is established in \cite[Lemma 2]{Ein.Ramification}.
 But when $B$ is globally generated and $W \subseteq \HH{0}{X}{A}$ consists of sections vanishing at some point $x \in X$, then $$\HH{0}{X}{B \otimes A^*} \otimes W \lra \HH{0}{X}{B}$$ cannot be surjective.
 \end{proof}

\begin{proof}[Proof of Theorem \ref{Intro.Thm.B}]
With $B$ as in \eqref{Possibilities.for.B}, we want to check that in fact
\[
\phi_{|A|}(X) \ = \ \Val( E_{|A|,B} ). 
\]
Let $W \subseteq \HH{0}{X}{A}$ be a basepoint-free subspace of codimension one (hence $\dim_\CC W = r$). Thanks to Lemma \ref{Key.Lemma.Section.1}, the issue is to show that multiplication
\[
\HH{0}{X}{B \otimes A^*} \otimes W \lra \HH{0}{X}{B}  \tag{*}
\]
is surjective except in the excluded cases. 

Suppose first that $B = \OO_X(K_X + (n+1)A)$ and that (*) is not surjective. By definition, this means that
\[  K_{0, n+1}(X, K_X; W) \ \ne \ 0. \]
We now apply Proposition \ref{Duality.for.Koszul.Groups} with $N = K_X$ and $s = r-1$. The hypotheses \eqref{Van.for.Duality.Eqn} follow from Kodaira vanishing, and we conclude that
\[
K_{r-n-1, 0}(X; W) \ \ne \ 0. 
\]
But $r - n -1 > 0$ since we assume that $X$ isn't a hypersurface, and we arrive at a contradiction to Lemma \ref{Elem.Koszul.Lemma|} (i).

The argument when $B = \OO_X(K_X + nA)$ proceeds along similar lines. Assume that (*) fails to be surjective. Duality applies thanks  to Kodaira and the assumption that $\HH{1}{X}{\OO_X} = 0$ when $n \ge 2$, and therefore
$K_{r-n-1,1}(X; W) \ne 0$. Using again that  $K_{r-n-1,0}(X; W) = 0$,  we conclude from the exact sequence in  Lemma \ref{Elem.Koszul.Lemma|}  that $K_{r-n-1, 1}(X; V) \ne 0$.  But then Green's Theorem \ref{Green.Kp1.Thm} implies that $X \subseteq \PP^r$ lies on an $(n+1)$-fold of minimal degree. 
\end{proof}

\begin{example}[Divisors in scrolls] 
\label{Divisors.in.Scrolls.Exampe}
To complete the picture we analyze the exceptional case in (ii) when $X \subseteq \PP^r$ sits as a divisor in an $(n+1)$-fold of minimal degree $Y \subseteq \PP^r$. For simplicity assume that $Y$ is smooth, so that $Y = \PP(Q)$ is the total space of an ample vector bundle $Q$ of rank $(n+1)$ on $\PP^1$. Write $q = \deg Q$, and denote by $H$ and $F$ respectively the classes of $\OO_{\PP(Q)}(1)$ and a fibre, so that $A = \OO_X(H)$. Then   $X \lin dH + eF $ for some integers $d \ge 2$ and $e$. Recalling that $K_Y \lin -(n+1)H + (q-2)F$, we see by adjunction that 
\[
\OO_X(K_X + nA) \ = \ \OO_X\big( (d-1)H + (e + q - 2)F \big).
\]
The coefficient of $H$ here being $< d$, this implies that
\begin{align*}
\HH{0}{X}{\OO_X(K_X + (n-1)A)} \ &= \ \HH{0}{Y}{\OO_Y\big( (d-2)H + (e + q - 2)F \big)} \\
\HH{0}{X}{\OO_X(K_X + nA)} \ &= \ \HH{0}{Y}{\OO_Y\big( (d-1)H + (e + q - 2)F \big)}.\end{align*}
Consequently the Steiner bundle $E_{|A|, B}$ doesn't vary with $X$. 
\end{example}

\appendix
\section{The theorem of Dolgachev--Kapranov}  

Let $X \subseteq \PP(V) = \PP^r $ be a finite set of $d \ge r+1$ points in linear general position, and denote by $\II = \II_X \subseteq \OO_{\PP(V)}$ the ideal sheaf of $X$. Green shows in \cite[Theorem (3.c.6)]{Koszul1} that $K_{r-2, 2}(\PP^r, \II; V) \ne 0$ if and only if $X$ lies  on a rational normal curve.\footnote{The statement in \cite{Koszul1} actually involves $K_{r-1,1}$ of the homogeneous   coordinate ring of $X$, but this is isomorphic to the stated group. See also \cite[Lemma 3.29]{AproduNagel} for another exposition. }  Given the arguments from the previous section, it is natural to expect that one can use this to get a new proof of the Torelli-type theorem of Dolgachev and Kapranov from \cite{Dolgachev.Kapranov} (along with the numerical improvements by Vall\`es \cite{Valles1}).     Inspired by some of the techniques in \cite{Koszul1}, we indicate here how this goes. For simplicity we assume that $r \ge 3$.

 Note to begin with that  each 
 $
 \HHstar{i} {\PP(V)}{\II} =_{\text{def}} \oplus_k\,  \HH{i}{\PP(V)}{\II(k)}
 $
 is a graded module over the symmetric algebra $\Sym(V)$. In particular, there is a natural map
\begin{equation} \label{H1.Module.Equation}
\HH{1}{\PP(V)}{\II} \otimes V \lra  \HH{1}{\PP(V)}{\II(1)} 
\end{equation}
On the other hand, every point $x \in  X\subseteq \PP(V)$ determines a dual hyperplane $H \subseteq \linser{V}$, and so $X$ itself gives rise to a normal crossing hyperplane arrangement $\Sigma H_i$ on $\linser{V}$. One checks that the Dolgachev--Kapranov bundle $E = \Omega_{\linser{V}}^1(\log \Sigma H_i)$ is the Steiner bundle on $\linser{V}$ determined by the multiplication map
 \[
 \HH{1}{\PP(V)}{\II(1)}^* \otimes V \lra  \HH{1}{\PP(V)}{\II}^* ,
 \]
deduced from \eqref{H1.Module.Equation}.  It is elementary that $X \subseteq \Val(E)$, and we want to verify
\begin{proposition} \label{Proposition.Appendix}
If $X \subsetneqq \Val(E)$, then $X$ lies on a rational normal curve in $\PP({V})$. 
\end{proposition} 
\noi Equivalently, fix a subspace $W \subseteq V$ of codimension one that generates $\OO_X$. In view of Lemma \ref{Key.Lemma.Section.1}, we need to show that if the mapping
\begin{equation} \label{W*.Equation}
\HH{1}{\PP(V)}{\II_X(1)}^* \otimes W \lra  \HH{1}{\PP(V)}{\II_X}^* \end{equation}
fails to be surjective, then $X$ lies on a rational normal curve. 

Since $W \subseteq V$, each $\HHstar{i}{\PP(V)}{\II}$ has the structure of a $\Sym(W)$-module.  The first point is that in bounded degrees, one can realize these  as the cohomology modules of a sheaf $\JJ$ on $\PP(W)$. Specifically: \begin{lemma} \label{Lemma.in.Appendix}
For any suitably large integer $k_0 \gg0$ one can construct a coherent sheaf $\JJ $  on $\PP(W)$ with the property that for $i < r-1 = \dim \PP(W)$ there are isomorphisms
\[
\HHstar{i}{\PP(V)}
{\II}_{\le k_0} \ \cong \ \HHstar{i}{\PP(W)}{\JJ}_{\le k_0}
\]
in degrees $\le k_0$, and these isomorphisms are compatible with the $\Sym(W)$-module structures on both sides. 
\end{lemma}
\vskip -5pt
 
Granting the Lemma for the time being, we give the 
\begin{proof} [Proof of Proposition \ref{Proposition.Appendix}] 
Tensoring the universal Koszul complex on $\PP(W)$ by $\JJ$, one arrives at a long exact sequence
\[
0 \lra \Lambda^r W \otimes \JJ \lra \Lambda^{r-1}W \otimes \JJ(1) \lra \Lambda^{r-2}W \otimes \JJ(2) \lra \ldots \lra \JJ(r)\lra 0
\]
of sheaves on $\PP(W)$. This in turn gives rise to a hypercohomology spectral sequence abutting to zero. The  bottom two rows of its $E_1$ page have the form
\vskip -15pt
\[
\xymatrix@C=12pt@R=5pt{
0 \ar[r] &\Lambda^r W \otimes H^1(\JJ)  \ar[r] &\Lambda^{r-1} W \otimes H^1(\JJ(1)) \ar[r] &\Lambda^{r-2} W \otimes H^1(\JJ(2)) \ar[r] &\Lambda^{r-3} W \otimes H^1(\JJ(3)) \ar[r] &
\\
0 \ar[r]&\Lambda^r W \otimes H^0(\JJ)  \ar[r] &\Lambda^{r-1} W \otimes H^0(\JJ(1)) \ar[r] &\Lambda^{r-2} W \otimes H^0(\JJ(2)) \ar[r] &\Lambda^{r-3} W \otimes H^0(\JJ(3)) \ar[r] &
}
\]
where the cohomology groups are taken on $\PP(W)$. The assumption \eqref{W*.Equation}
means (thanks to the Lemma) that the map 
\[  H^1(\II) \, = \, \Lambda^r W \otimes H^1(\JJ)  \lra \Lambda^{r-1} W \otimes H^1(\JJ(1)) \, = \, W^* \otimes  H^1(\II(1))\]  has a non-trivial kernel. This must cancel against the $E_2$ term coming from the bottom row of the spectral sequence. In other words,
\[    K_{r-2,2}(\PP(W), \JJ) \ \ne \ 0. \]
But $K_{r-2,1}(\PP(W), \JJ) = 0$ since $X$ does not lie on a hyperplane, so by (the analogue of) Lemma \ref{Elem.Koszul.Lemma|} (ii), we conclude that $K_{r-2,2}(\PP(V), \II) \ne 0$. Then Green's theorem applies to put $X$ on a rational normal curve. 
\end{proof}

\begin{proof}[Proof of Lemma \ref{Lemma.in.Appendix}]
Let $w \in \PP(V)$ be the point corresponding to $W \subseteq V$, so that in particular $w \not \in X$. Projection 
$
\pi : \big( \PP(V) - \{ w \}\big) \lra \PP(W)
$
from $w$ gives an identification
\[   \PP(V) - \{ w \} \ \cong \ \Spec_{\PP(W)}\Big(\Sym  \big( \OO_{\PP(W)} \oplus \OO_{\PP(W)}(-1) \big) \Big). \]
Then for $k_0 \gg 0$, one arrives at a surjective mapping
\[
\eps : \Sym^{k_0}  \big( \OO_{\PP(W)} \oplus \OO_{\PP(W)}(-1) \big) \lra \pi_* \OO_X
\]
of sheaves on $\PP(W)$. It suffices to take $\JJ = \ker (\eps)$. 
\end{proof}

\end{document}